\documentclass{amsart}
\usepackage{amsmath,amssymb,bm,color}
\usepackage{amsmath,amsthm,amssymb,cases}
\usepackage{amscd}
\usepackage[dvips]{graphicx}

\theoremstyle{definition}
\newtheorem{defn}{Definition}[section]
\newtheorem{rem}[defn]{Remark}

\theoremstyle{plain}
\newtheorem{thm}[defn]{Theorem}
\newtheorem{prop}[defn]{Proposition}
\newtheorem{lem}[defn]{Lemma}
\newtheorem{cor}[defn]{Corollary}

\newtheorem{sublem}[defn]{Sublemma}

\numberwithin{equation}{section}

\title[Uniform hyperbolicity for curve graphs of non-orientable surfaces]{Uniform hyperbolicity for curve graphs of non-orientable surfaces}

\author[E.~Kuno]{Erika Kuno}
\address{
(Erika Kuno)
Department of Mathematics,
Tokyo Institute of Technology,
2-12-1 Ohokayama, Meguro-ku, Tokyo 152-8551, Japan
}
\email{kuno.e.aa@m.titech.ac.jp}
\date{\today}

\begin{document}
\maketitle

\begin{abstract}
Hensel-Przytycki-Webb proved that all curve graphs of orientable surfaces are $17$-hyperbolic. In this paper, we show that curve graphs of non-orientable surfaces are $17$-hyperbolic by applying  Hensel-Przytycki-Webb's argument. We also show that arc graphs of non-orientable surfaces are $7$-hyperbolic, and arc-curve graphs of (non-)orientable surfaces are $9$-hyperbolic.
\end{abstract}

\tableofcontents

\section{Introduction}
Geometric group theory is a new field investigating structures of groups from a geometric viewpoint. In this field, it is one of the most important ideas to consider finitely generated groups themselves as geometric objects called Cayley graphs.
Geometric group theory is related to a lot of mathematic fields, for example, low-dimensional topology, hyperbolic geometry, algebraic topology. 
The quasi-isometry classification of finitely generated groups becomes one of the main research themes in geometric group theory since a suggestion by Gromov in the 1980s.
Hence, the study of quasi-isometry invariants, namely properties of spaces or groups that are invariant under quasi-isometries, is very important. 
In particular, the notion of {\it Gromov hyperbolicity} is one of the quasi-isometry invariants.
Furthermore, some quasi-isometry invariants arise from Gromov hyperbolicity.
Therefore, investigating whether geodesic spaces and finitely generated groups are Gromov hyperbolic or not is quite important for geometric group theory. 

For $g\geq 1$ and $n\geq 0$, let $N=N_{g,n}$ be a compact connected non-orientable surface of genus $g$ with $n$ boundary components.
The {\it curve graph} $\mathcal{C}(N)$ of $N$ is the graph whose vertex set is the set of homotopy classes of essential simple closed curves (or curves) and whose edges correspond to disjoint curves.
Curve graphs are often used to study mapping class groups of surfaces, geometric group theory, hyperbolic geometry, and so on.
In this paper, we consider a graph as a geodesic space as follows.
We set the length of each edge to be one, and the distance between two vertices is the length of the shortest edge-path connecting them.
A triangle formed by geodesic edge-paths in the graph (we call such a triangle a {\it geodesic triangle}) has a $k$-{\it center} ($k\geq 0$) if there exists a vertex such that the distance from it to each side of $T$ is not more than $k$.
A connected graph is {\it $k$-hyperbolic} if every geodesic triangle in the graph has a $k$-center.
We say that a graph is ({\it Gromov}) {\it hyperbolic} if it is $k$-hyperbolic for some $k\geq0$, and we refer to such a constant $k$ as a hyperbolicity constant for the graph.
Bestvina-Fujiwara\cite{BF} first proved that $\mathcal{C}(N)$ is Gromov hyperbolic, and Masur-Schleimer\cite{MS} gave another proof. However, the uniform hyperbolicity for curve graphs of non-orientable surfaces was not known. The main result of this paper is to prove the following:

\begin{thm}\label{mainthm}
If $\mathcal{C}(N)$ is connected, then it is {\rm 17}-hyperbolic.
\end{thm}

Let $S=S_{g,n}$ be an orientable surface of genus $g\geq 0$ with $n\geq 0$ boundary components. First, Masur-Minsky\cite{MM} proved that each curve graph $\mathcal{C}(S)$ of $S$ is hyperbolic in 1999.
After their original proof, various other proofs of hyperbolicity for curve graphs of orientable surfaces were given by several authors. 
Bowditch\cite{B1} gave an upper bound of the hyperbolicity constant which depends on the genus and the number of boundary components in 2006.
Another proof was given by Hamenst\"{a}dt\cite{H07}. 
Recently, Aougab\cite{A}, Bowditch\cite{B2}, Clay-Rafi-Schleimer\cite{CRS}, and Hensel-Przytycki-Webb\cite{HPW} independently proved that one can choose the hyperbolicity constants which do not depend on the topological types of orientable surfaces. 
In particular, Hensel-Przytycki-Webb\cite{HPW} showed that $\mathcal{C}(S)$ is $17$-hyperbolic by a combinatorial argument.
The argument by Hensel-Przytycki-Webb seems to give an optimum constant.

We prove Theorem~\ref{mainthm} by applying Hensel-Przytycki-Webb's argument in~\cite{HPW} to the case of non-orientable surfaces. 

In~\cite{HPW}, they also showed that arc graphs of orientable surfaces are $7$-hyperbolic. We prove a similar result for non-orientable surfaces:

\begin{thm}\label{second_thm}
An arc graph $\mathcal{A}(N)$ of $N$ is {\rm 7}-hyperbolic.
\end{thm}

We also consider arc-curve graphs.
The hyperbolicity for arc-curve graphs of orientable surfaces was proved by Korkmaz-Papadopoulos\cite[Corollary 1.4]{KP}. 
The uniform hyperbolicity, however, was not shown. 
We also prove:

\begin{thm}\label{third_thm}
If an arc-curve graph $\mathcal{AC}(N)$ of $N$ is connected, then it is {\rm 9}-hyperbolic.
\end{thm}

By the same argument as we give in the proof of Theorem~\ref{third_thm}, we prove the following:

\begin{thm}\label{fourth_thm}
If an arc-curve graph $\mathcal{AC}(S)$ of $S$ is connected, then it is {\rm 9}-hyperbolic.
\end{thm}

In~\cite{HPW}, for the cases where $a$, $b$, and $d$ are vertices of $\mathcal{A}(S)$ and where $a$, $b$, and $d$ are vertices of $\mathcal{C}(S)$, Hensel-Przytycki-Webb proved a geodesic triangle $T=abd$ has a $7$-center and a $9$-center in $\mathcal{AC}(S)$ respectively. We show that a geodesic triangle $T=abd$ has an $8$-center for the cases where $a$ is a vertex of $\mathcal{C}(S)$ and $b$ and $d$ are vertices of $\mathcal{A}(S)$, and where $a$ and $b$ are vertices of $\mathcal{C}(S)$ and $d$ is a vertex of $\mathcal{A}(S)$ to prove Theorem~\ref{fourth_thm}.
 
Here, we describe our idea of the proof of Theorem~\ref{mainthm}.
First, in Section~\ref{Unicorn paths} we define {\it unicorn arcs} and {\it unicorn paths} between two arcs on $N$, which are defined in~\cite{HPW} for the case of orientable surfaces. 
One of the important properties of unicorn paths is that they are paths in each arc graph $\mathcal{A}(N)$ of $N$ (Proposition~\ref{rem1}).
 
Second, we show key lemmas related to unicorn paths to prove Theorem~\ref{mainthm}.
The particularly important lemma states that these paths form 1-slim triangles in $\mathcal{A}(N)$ (Lemma~\ref{key2}). 

Finally, in Section~\ref{Curve graphs are hyperbolic}, we show the following. For any geodesic triangle $T=abd$ in $\mathcal{C}(N)$ ($a$, $b$, and $d$ are three vertices of $\mathcal{C}(N)$), let $\bar{a}$, $\bar{b}$, and $\bar{d}$ be three vertices of $\mathcal{A}(N)$ which are adjacent to $a$, $b$, and $d$ in the arc-curve graph $\mathcal{AC}(N)$ of $N$ respectively.
Then, the distance between the side of $T$ connecting $a$ and $b$ (resp. $b$ and $d$, $d$ and $a$) and any unicorn arc obtained from $\bar{a}$ and $\bar{b}$ (resp. $\bar{b}$ and $\bar{d}$, $\bar{d}$ and $\bar{a}$) is at most $8$.
Therefore, we can prove that $T$ has a $9$-center in $\mathcal{AC}(N)$.
Furthermore, we construct a retraction $r\colon\mathcal{AC}(N)\rightarrow \mathcal{C}(N)$, and show that $r$ is $2$-Lipschitz (Lemma~\ref{r_2-Lipschitz}).
When we prove this, there is a greatly different point from the case of orientable surfaces: if an arc $a$ goes through ``crosscaps" odd number of times, then $r(a)$ is ``twised."
After having proved this, we see that a $9$-center in $\mathcal{AC}(N)$ of $T$ is mapped to a $17$-center of $T$ in $\mathcal{C}(N)$.
This gives a proof of Theorem~\ref{mainthm}.

\section{Preliminaries}

A compact connected {\it non-orientable surface} of genus $g\geq 1$ with $n\geq 0$ boundary components is the connected sum of $g$ projective planes which is removed $n$ open disks.
We denote it by $N=N_{g,n}$. Note that $N_{g,n}$ is homeomorphic to the surface obtained from a sphere by removing $g+n$ open disks and attaching $g$ M\"{o}bius bands along their boundaries (see the left-hand side of Figure~\ref{fig_two_pattern_nonori_surface}).
We represent $N_{g,n}$ as a sphere with $g$ {\it crosscaps} and $n$ boundary components (see the right-hand side of Figure~\ref{fig_two_pattern_nonori_surface}).
We identify antipodal points of each periphery of a crosscap.

\begin{figure}[h]
\includegraphics[scale=0.63]{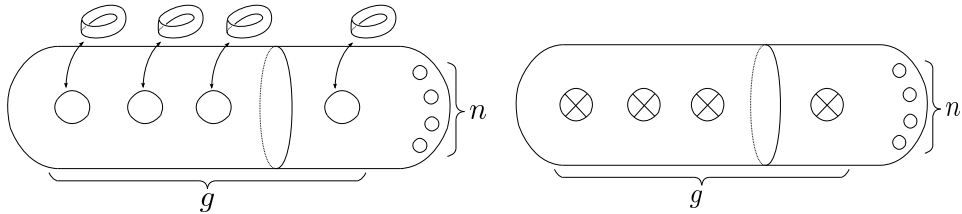}
\caption{A non-orientable surface $N_{g,n}$.}\label{fig_two_pattern_nonori_surface}
\end{figure}

An arc $a$ on $N$ is {\it properly embedded} if $\partial a \subseteq \partial N$ and $a$ is transversal to $\partial N$.
An arc $a$ on $N$ is called {\it essential} if it is not homotopic into $\partial N$.
A curve on $N$ is called {\it essential} if it does not bound a disk or a M\"{o}bius band, and it is not homotopic to a boundary component of $N$. 
We remark that a homotopy fixes each boundary component of $N$ setwise.
From now on, we consider arcs and curves which are properly embedded and essential.
The {\it arc-curve graph} $\mathcal{AC}(N)$ of $N$ is the graph whose vertex set $\mathcal{AC}^{(0)}(N)$ is the set of homotopy classes of arcs and curves on $N$.
Two vertices form an edge if they can be represented by disjoint arcs or curves.
The {\it arc graph} $\mathcal{A}(N)$ of $N$ is the subgraph induced on the vertices that are homotopy classes of arcs on $N$.
The {\it curve graph} $\mathcal{C}(N)$ of $N$ is the subgraph induced on the vertices that are homotopy classes of curves on $N$.

We set the length of each edge in $\mathcal{AC}(N)$, $\mathcal{A}(N)$, and $\mathcal{C}(N)$ to be $1$.
We define the distances $d_{\mathcal{AC}}(\cdot, \cdot)$, $d_{\mathcal{A}}(\cdot, \cdot)$, and $d_{\mathcal{C}}(\cdot, \cdot)$ in $\mathcal{AC}(N)$, $\mathcal{A}(N)$, and $\mathcal{C}(N)$ respectively by the minimal length of sequences of edges connecting the two vertices.
Now, we consider $\mathcal{AC}(N)$, $\mathcal{A}(N)$, and $\mathcal{C}(N)$, as geodesic spaces.

Two arcs $a$, $b$ (or two curves $a$, $b$) on $N$ are in {\it minimal position} if the number of intersections between $a$ and $b$ is minimal in the homotopy classes of $a$ and $b$. 

\begin{prop}\label{minimal_position}
Two arcs $a$, $b$ on $N$ are in minimal position if and only if $a$ and $b$ intersect transversely and they do not form any {\it bigons} (i.e. an embedded disk on $N$ bounded by a subarc of $a$ and a subarc of $b$) or any {\it half-bigons} (i.e. an embedded disk on $N$ bounded by a subarc of $a$, a subarc of $b$, and a part of a boundary component of $N$).
\end{prop}

We use the following proposition to prove Proposition~\ref{minimal_position}.

\begin{prop}{\rm(}\cite[Proposition 2.1]{Stukow}{\rm)}\label{Stukow}
Let $N$ be a smooth, non-orientable, compact surface, and $a$ and $b$ essential curves on $N$. Then $a$ and $b$ are in minimal position if and only if $a$ and $b$ do not form a bigon.
\end{prop}

\begin{proof}[Proof of Proposition~\ref{minimal_position}]
If $a$ and $b$ bound bigons or half-bigons, then we can reduce intersection points by a homotopy through bigons or half-bigons. 

Conversely, suppose that two arcs $a$ and $b$ on $N$ are not in minimal position.
We collect the boundary components which have endpoints of $a$ and $b$ in one side by a homeomorphism preserving intersections between $a$ and $b$.
We make a mirror reflective surface $N'$ of $N$ for the side which have endpoints of $a$ and $b$, and assume that $a'$ and $b'$ are arcs on $N'$ corresponding to $a$ and $b$ on $N$ respectively.
Note that $a'$ and $b'$ are not in minimal position since $a$ and $b$ are not in minimal position.
We attach each boundary component of $N'$ which has the endpoints of $a'$ and $b'$ to the reflective part of $N$, and let $M$ be the resulting surface.
Then $a\cup a'$ and $b\cup b'$ are essential curves and not in minimal position on $M$. By Proposition~\ref{Stukow}, $a\cup a'$ and $b\cup b'$ form bigons.
From the assumption that $N$ and $N'$ are mirror reflective surfaces each other, we have the following two cases.
One is that $a$ and $b$ form bigons on $N$ and $a'$ and $b'$ also form bigons on $N'$ at the reflective parts. 
The other is that $a\cup a'$ and $b\cup b'$ form bigons on $M$ which are mirror reflective for attached parts.
The former implies that $a$ and $b$ form bigons on $N$, and the latter implies that $a$ and $b$ form half-bigons on $N$, as desired. 
\end{proof}

\section{Unicorn paths}\label{Unicorn paths}

In this section, all lemmas come from Section 3 in~\cite{HPW} by changing the assumption of orientable surfaces to non-orientable surfaces. 

In this paper, we denote by $\overline{\alpha \alpha '}_{a}$ the subarc of $a$ whose endpoints are $\alpha $ and $\alpha '$.

\begin{defn}
Let $a$ and $b$ be two arcs on $N$ which are in minimal position, and let $\alpha$ and $\beta$ be one of the endpoints of $a$ and $b$ respectively. Choose $\pi \in a\cap b$. Let $a^{\prime}$ be a subarc of $a$ whose endpoints are $\alpha$ and $\pi $, and $b^{\prime}$ a subarc of $b$ whose endpoints are $\beta$ and $\pi $. If $a^{\prime}\cup b^{\prime}$ is an embedded arc on $N$, we say that $a^{\prime}\cup b^{\prime}$ is a \it{unicorn arc} obtained from $a^{\alpha}$, $b^{\beta}$ and $\pi $.
\end{defn}

A unicorn arc is uniquely determined by $\pi$, although not all intersection points between $a$ and $b$ determine unicorn arcs since the resulting arcs may not be embedded on $N$.

Note that $a^{\prime}\cup b^{\prime}$ is an essential arc.
Indeed, if $a^{\prime}\cup b^{\prime}$ is not essential, that is, if $a^{\prime}\cup b^{\prime}$ is homotopic into a boundary component of $N$, then $a$ and $b$ form a half-bigon.
This contradicts the assumption that $a$ and $b$ are in minimal position. 

\begin{defn}
Let $a'\cup b'$, $a''\cup b''$ be two unicorn arcs obtained from $a^{\alpha}$ and $b^{\beta}$, where $a', a''\subset a$ and $b', b''\subset b$.
We define $a'\cup b'\leq a''\cup b''$ by $a''\subset a'$ and $b'\subset b''$.
\end{defn}

\begin{lem}
The relation $\leq$ is a total order.
\end{lem} 

\begin{proof}
We have $a'\cup b'\leq a'\cup b'$ since $a'\subset a' $ and $b'\subset b'$.
Suppose that $a_{1}\cup b_{1}\leq a_{2}\cup b_{2}$ and $a_{2}\cup b_{2}\leq a_{3}\cup b_{3}$.
Then $a_{3}\subset a_{2}\subset a_{1}$ and $b_{1}\subset b_{2}\subset b_{3}$, and so it follows that $a_{1}\cup b_{1}\leq a_{3}\cup b_{3}$.
Suppose that $a_{1}\cup b_{1}\leq a_{2}\cup b_{2}$ and $a_{2}\cup b_{2}\leq a_{1}\cup b_{1}$.
Then we have $a_{2}\subset a_{1}$ and $b_{1}\subset b_{2}$, and $a_{1}\subset a_{2}$ and $b_{2}\subset b_{1}$.
Therefore $a_{1}\cup b_{1}= a_{2}\cup b_{2}$ .
For unicorn arcs $c_{1}$ and $c_{2}$ obtained from $a^{\alpha}$ and $b^{\beta}$, set $c_{1}=a_{1}\cup b_{1}$ and $c_{2}=a_{2}\cup b_{2}$.
Since both $a_{1}$ and $a_{2}$ contain $\alpha$, we have either $a_{1}\subset a_{2}$ or $a_{2}\subset a_{1}$.
We assume that $a_{1}\subset a_{2}$, and take $\pi_{1}\in a\cap b$ such that $a_{1}=\overline{\alpha \pi_{1}}_{a}$.
Then $\pi_{1}\not\in b_{2}$, since $c_{2}$ is an embedded arc.
Hence $b_{2}$ is contained in one of the components of $b-\{\pi_{1}\}$ for the connectedness of $b_{2}$.
Since $b_{1}$ is one of the components of $b-\{\pi_{1}\}$ which has $\beta$ and $b_{2}$ has $\beta$, we get $b_{1}\subset b_{2}$.
Hence $c_{2}\leq c_{1}$, and so the relation $\leq$ is a total order.
\end{proof}

Let $(c_{1}, c_{2},\ldots, c_{n-1})$ be the ordered set of all unicorn arcs obtained from $a^{\alpha}$ and $b^{\beta}$.

\begin{defn}
We call the sequence ${\mathcal P}(a^{\alpha },b^{\beta })=(a=c_{0}, c_{1},\ldots, c_{n-1}, c_{n}=b)$ the {\it unicorn path} between $a^{\alpha}$ and $b^{\beta}$.
\end{defn}

Then, we have a natural question similar to that of the case of orientable surfaces whether a unicorn path ${\mathcal P}(a^{\alpha}, b^{\beta})$ becomes a path in $\mathcal{A}(N)$ or not. We can show the following:

\begin{prop}\label{rem1}
Consecutive arcs in a unicorn path represent adjacent vertices of $\mathcal{A}(N)$.
\end{prop}

\begin{proof}
Let $c_{i}= a'\cup b'$ ($2\leq i\leq n-1$) and $\pi \in a'\cap b'$.
We assume that $\pi'$ is the point in $(a-a')\cap b$ which is nearest to $\alpha$ along $a$ of the points determining a unicorn arc.
Therefore, the intersection point $\pi'$ determines the unicorn arc $c_{i-1}$.
The unicorn arc $c_{i}$ does not pass any points between $\pi$ and $\pi'$ of $a\cap b$, otherwise the point becomes the next point determining the unicorn arc next to $c_{i}$ and this contradicts the assumption of $\pi'$.
Thus, $c_{i}$ and $c_{i-1}$ do not intersect between $\pi$ and $\pi'$.
Furthermore, there exists an arc homotopic to $c_{i}$ which is disjoint from $c_{i-1}$.
Indeed, it is sufficient to choose the neighborhood of $a'$ not intersecting $c_{i-1}$ when $c_{i}$ turns at $\pi$, and the neighborhood of $b'$ not intersecting $c_{i-1}$ at $\pi'$.
For $i=1, n$, the fact that $c_{i-1}$ and $c_{i}$ form an edge follows similarly.
\end{proof}

Especially, we deduce that all arc graphs are connected by the existence of unicorn paths.

\begin{cor}
$\mathcal{A}(N)$ is connected.
\end{cor}

\begin{lem}{\rm(cf.} \cite[Lemma 3.3]{HPW}{\rm)}\label{key1}
Let $a$, $b$, and $d$ be three arcs on $N$ which are mutually in minimal position, and let $\alpha$, $\beta$, and $\delta$ be one of the endpoints of $a$, $b$, and $d$. For each $c\in \mathcal{P}(a^{\alpha} ,b^{\beta } )$, there exists $c^{*}\in \mathcal{P}(a^{\alpha} ,d^{\delta}) \cup \mathcal{P}(b^{\beta}, d^{\delta})$, such that $c$, $c^{*}$ represent adjacent vertices of $\mathcal{A}(N)$. 
\end{lem}

\begin{proof}
For any $c\in \mathcal{P}(a^{\alpha} ,b^{\beta } )$, let $c=a'\cup b'$.
If $c\cap d=\emptyset$, then we take $c^{*}=d$. 
When $c\cap d\not=\emptyset$, we assume that $d'$ is the maximal subarc of $d$ with endpoint $\delta $ whose interior is disjoint from $c$, and $\sigma$ is the other endpoint of $d'$.
Thus $d'=\overline{\sigma \delta }_{d}$
Then $\sigma \in a'$ or $\sigma \in b'$, and without loss of generality, we can assume that $\sigma \in a'$.
By taking $c^{*}=\overline{\alpha \sigma }_{a}\cup \overline{\sigma \delta }_{d}$, we see that $c^{*}$ and $c$ represent adjacent vertices of $\mathcal{A}(N)$.
\end{proof}

Note that $c$ and $d$ may not be in minimal position.

\begin{lem}{\rm(cf.} \cite[Lemma 3.4]{HPW}{\rm)}\label{key2}
Let $a$, $b$, and $d$ be three arcs on $N$ which are mutually in minimal position, and let $\alpha$, $\beta$, and $\delta$ be one of the endpoints of $a$, $b$, and $d$. Then there exist $c^{1}\in\mathcal{P}(a^{\alpha} ,b^{\beta})$, $c^{2}\in\mathcal{P}(b^{\beta}, d^{\delta})$, and $c^{3}\in\mathcal{P}(d^{\delta}, a^{\alpha})$ such that $c^{i}$ and $c^{j}$ $(i\not=j, i, j=1,2,3)$ represent adjacent vertices of $\mathcal{A}(N)$.
\end{lem}

\begin{proof}
First, suppose that two of $a$, $b$, $d$ are disjoint.
Without loss of generality, we can assume that $a$ and $b$ are disjoint. Let $d'=\overline{\delta \pi }_{d}$ be the maximal subarc of $d$ whose interior is disjoint from $a\cup b$. 
Then $\pi \in a$ or $\pi \in b$, and here we assume that $\pi \in a$.
It is sufficient to take $c^{1}=a$, $c^{2}=b$, and $c^{3}=\overline{\delta\pi}_{d}\cup \overline{\pi\alpha}_{a}$. 
Otherwise, that is, when any two of $a$, $b$, $d$ intersect transversely, for any unicorn arc $c_{i}\in\mathcal{P}(a^{\alpha} ,b^{\beta } )$ ($0\leq i\leq n-1$), denote by $d_{i}=\overline{\pi _{i}\delta }_{d}$ the subarc of $d$ whose interior is disjoint from $c_{i}$.
Set $c_{i}=a_{i}\cup b_{i}$.
Then $\pi _{i}\in a_{i}$ or $\pi _{i}\in b_{i}$.
Here we assume that $\pi _{i}\in a_{i}$. 
In the case where $b_{i+1}$ and the interior of $d_{i}$ are not disjoint, let $\varepsilon$ be the intersection point between $d_{i}$ and $b_{i+1}$ which is closest to $\delta$ along $d$.
Then we take $c^{1}=c_{i}$, $c^{2}=\overline{\beta\varepsilon}_{b}\cup \overline{\varepsilon\delta}_{d}$, and $c^{3}=\overline{\delta\pi _{i}}_{d}\cup \overline{\pi _{i}\alpha }_{a}$.
In the case where $b_{i+1}$ and the interior of $d_{i}$ are disjoint, let $\sigma $ be the intersection point between $(d-{\rm Int}(d_{i}))$ and $c_{i+1}$ which is nearest to $\pi _{i}$ along $d$, where Int($d_{i}$) is the interior of $d_{i}$. 
If $\sigma \in a_{i+1}$, then we go back to the beginning of this proof changing $i$ to $i+1$, since we can not take three arcs satisfy the statement of Lemma~\ref{key2}.
If $\sigma \in b_{i+1}$, then let $\pi '$ be the intersection point between $\overline{\sigma \delta }_{d}$ and $a_{i}$ which is closest to $\alpha $ along $a$.
Then we take $c^{1}=c_{i+1}$, $c^{2}=\overline{\beta \sigma }_{b}\cup \overline{\sigma\delta}_{d}$, and $c^{3}=\overline{\delta \pi _{i}}_{d}\cup \overline{\pi _{i}\alpha}_{a}$. 
Finally, we have to consider the case where $(d-{\rm Int}(d_{i}))\cap c_{i+1}$ is empty.
Let $\pi '$ be the intersection point between $d$ and $a_{i}$ which is closest to $\alpha$ along $a$. Then we take $c^{1}=c_{i+1}$, $c^{2}=d$, and $c^{3}=\overline{\delta \pi'}_{d}\cup \overline{\pi '\alpha}_{a}$, and so we are done.
\end{proof}

We now prove that unicorn paths are invariant under taking subpaths, up to one exception.

\begin{lem}{\rm(cf.} \cite[Lemma 3.5]{HPW}{\rm)}\label{key3}
For every $0\leq i< j\leq n$, either $\mathcal{P}(c_{i}^{\alpha}, c_{j}^{\beta})$ is a subpath of $\mathcal{P}(a^{\alpha}, b^{\beta})$,
or $c_{i}$, $c_{j}$ represent adjacent vertices of $\mathcal{A}(N)$ when $j=i+2$.
\end{lem}

Before we prove Lemma~\ref{key3}, we need the following.

\begin{sublem}{\rm(cf.} \cite[Sublemma 3.6]{HPW}{\rm)}\label{sub_key3}
Let $\alpha$ and $\alpha '$ be the endpoints of $a$.
Let $c=c_{n-1}\in \mathcal{P}(a^{\alpha}, b^{\beta})$, which means that $c=a'\cup b'$ with the interior of $a'$ disjoint from $b$.
Let $\tilde{c}$ be the arc homotopic to $c$ obtained by homotopying $a'$ slightly off $a$ in the direction toward $\beta$ so that $a'\cap \tilde{c}=\emptyset$.
Then either $\tilde{c}$ and $a$ are in minimal position, or they bound exactly one half-bigon shown in Figure~\ref{fig_subkey3} (the shaded region is the half bigon).
In that case, after homotopying $\tilde{c}$ through that half-bigon to $\bar{c}$, the arcs $\bar{c}$ and $a$ are already in minimal position.   
\end{sublem}

\begin{figure}[h]
\includegraphics[scale=0.90]{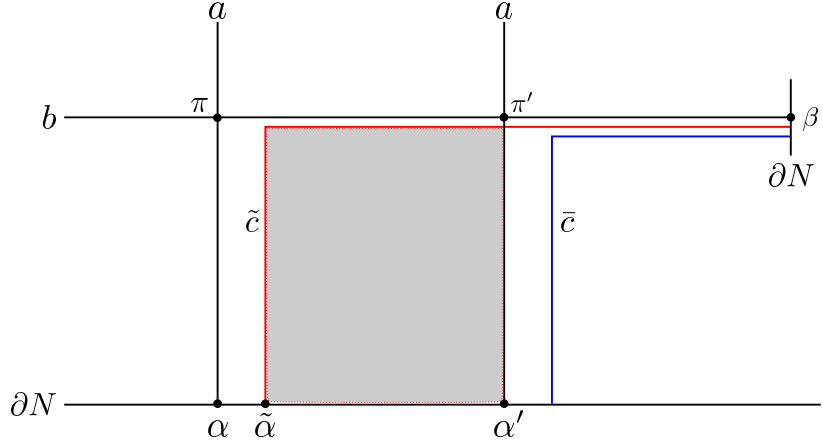}
\caption{The only possible half-bigon between $\tilde{c}$ and $a$.}\label{fig_subkey3}
\end{figure}

\begin{proof}[Proof of Sublemma~\ref{sub_key3}]
Let $\tilde{\alpha}$ be the endpoint of $\tilde{c}$ corresponding to $\alpha $ of $c$. 
When $\tilde{c}$ and $a$ are not in minimal position, $\tilde{c}$ and $a$ bound bigons or half-bigons. 
If $\tilde{c}$ and $a$ bound a bigon, then $a$ and $b$ also bound the bigon. 
This contradicts the assumption that $a$ and $b$ are in minimal position. 
Therefore, $\tilde{c}$ and $a$ do not bound any bigons but a half-bigon $\tilde{c}'a''$, where $\tilde{c}'\subset \tilde{c} $ and $a''\subset a'$. 
Let $\pi'=\tilde{c}'\cap a''$. 
The endpoint of $\tilde{c}'$ which is distinct from $\pi'$ is $\tilde{\alpha }$. 
Indeed, assume that the endpoint of $\tilde{c}'$ is $\beta $. 
Then $a$ and $b$ form a half-bigon using $\beta$, one of the endpoints of $a$, and $\pi'\in a\cap b$. 
This contradicts the assumption that $a$ and $b$ are in minimal position. 
On the other hand, the endpoint of $a''$ which is distinct from $\pi'$ is $\alpha'$. 
Indeed, assume that the endpoint of $a''$ is $\alpha $. 
Then $\tilde{c}'=\overline{\tilde{\alpha }\pi'}_{\tilde{c} }$ and $a''=\overline{\pi'\alpha}_{a}$ form a half-bigon and $\pi\in a\cap b$ is contained in $a''$.
Hence, $a$ and $b$ form a bigon, and this contradicts the assumption that $a$ and $b$ are in minimal position.  
The interior of $a''$ is disjoint from $b$, since the interior of $a'$ is disjoint from $b$, and since $a$ and $b$ are in minimal position. 
Moreover, $\pi $ and $\pi '$ are consecutive intersection points with $a$ on $b$. 
Hence, $\tilde{c}$ and $a$ bound exactly one half-bigon shown in Figure~\ref{fig_subkey3}. 

Let $b''$ be the component of $b-\{\pi '\}$ containing $\beta$, that is, set $b''=\overline{\pi '\beta }_{b}$. 
Let $\bar{c}$ be an arc obtained from $a''\cup b''$ by homotopying it slightly off $a''$ in the direction toward $\beta$.   
Since the endpoint of $a''$ which is distinct from $\pi'$ is $\alpha'$ and the interior of $a''$ is disjoint from $b$, the condition of $\bar{c}$ is the same as that of $\tilde{c}$. 
Applying to $\bar{c}$ the same argument as to $\tilde{c}$, but with the endpoint of $a$ interchanged, it follows that either $\bar{c}$ and $a$ are in minimal position, or they bound exactly one half-bigon $\bar{c}'a'''$, where $\bar{c}'\subset \bar{c}$ and $a'''\subset a$. 
In the latter case, we get $a'\subset a'''$, in particular $\pi \in a'''$. 
This contradicts the fact that the interior of $a'''$ should be disjoint from $b$. 
Since $\pi$ and $\pi'$ are consecutive intersection points with $a$ on $b$, $\bar{c}$ is homotopic to $\tilde{c}$, and so $\bar{c}$ and $\tilde{c}$ are representatives of the same element in $\mathcal{A}(N)$, as desired.  
\end{proof}

\begin{proof}[Proof of Lemma~\ref{key3}]
We can assume that $i=0$ and $j=n-1$. Hence, $c_{i}=c_{0}=a$ and $c_{j+1}=c_{n}=b$. 
We set $c_{j}=a'\cup b'$ ($=c_{n-1}$), where $a'$ and $b'$ are subarcs of $a$ and $b$. 
Then we see that $a'$ intersects $b$ only once at its endpoint $\pi $ distinct from $\alpha$.  
Let $\tilde{c}$ be the arc obtained from $c_{j}$ as in Sublemma~\ref{sub_key3}, and $\beta'$ the other endpoint of $b$. 
We note that the points of $a\cap b$ on $\overline{\pi \beta'}_{b}$ do not determine any unicorn arcs obtained from $a^{\alpha}$ and $b^{\beta}$. 

When $\tilde{c}$ and $a$ are in minimal position, the points of $(a\cap b)-\{\pi \}$ determining unicorn arcs in $\mathcal{P}(a^{\alpha }, b^{\beta })$ give all unicorn arcs in $\mathcal{P}(a^{\alpha }, \tilde{c}^{\beta } )$, since $a\cap \tilde{c}$ is coincident with $(a\cap \overline{\pi \beta'}_{b})-\{\pi \}$. Hence, in this case, $\mathcal{P}(c_{0}^{\alpha}, c_{n-1}^{\beta})$ becomes a subpath of $\mathcal{P}(a^{\alpha}, b^{\beta})$.

Suppose that $\tilde{c}$ and $a$ are not in minimal position. 
Let $\bar{c}$ be the arc from Sublemma~\ref{sub_key3} which is homotopic to $c_{j}$ and in minimal position with $a$. 
Let $\pi'$ be the point of $a\cap b$ with the same setting as in Sublemma~\ref{sub_key3}. 
Let $a''=\overline{\alpha\pi'}_{a}$ and $b''=\overline{\pi' \beta}_{b}$. 
We set $a^{*}=a-a''$. 
Suppose that $a^{*}$ and $b''$ intersect outside of $\pi'$. 
The points of $(a\cap b)-\{\pi, \pi'\}$ determining unicorn arcs in $\mathcal{P}(a^{\alpha }, b^{\beta })$ give all unicorn arcs in $\mathcal{P}(a^{\alpha }, \bar{c}^{\beta})$, since $a\cap \bar{c}$ is coincident with $(a\cap \overline{\pi' \beta}_{b})-\{\pi' \}$. 
Hence, in this case, $\mathcal{P}(c_{0}^{\alpha}, c_{n-1}^{\beta})$ becomes a subpath of $\mathcal{P}(a^{\alpha}, b^{\beta})$.
Suppose that $a^{*}\cap b=\{\pi' \}$. 
Then, $c_{0}=a$, $c_{1}=\overline{\alpha \pi '}_{a}\cup \overline{\pi '\beta }_{b}$, $c_{2}=\overline{\alpha \pi }_{a}\cup \overline{\pi \beta}_{b}$, and $c_{3}=b$ are all unicorn arcs obtained from $a^{\alpha}$ and $b^{\beta}$. 
Then, we get $a=c_{0}$ and $\bar{c}$ are disjoint.
Recall that $\bar{c}$ is homotopic to $c_{j}$ (now it follows that $j=2$).
Hence, $c_{0}$ and $c_{2}$ represent adjacent vertices of $\mathcal{A}(N)$, as desired.
\end{proof}

\begin{rem}
Slightly abusing the notation, we consider vertices $a$, $b$ of $\mathcal{A}(N)$, $\mathcal{C}(N)$, and $\mathcal{AC}(N)$ as arcs or curves $a$, $b$ on $N$ which are in minimal position from now on.
\end{rem}

\section{Arc graphs are uniformly hyperbolic}

\begin{defn}
We define the following family $P(a, b)$ of unicorn paths to a pair of vertices $a$, $b$ in $\mathcal{A}(N)$.
Let $(a,b)$ be an edge in $\mathcal{A}(N)$ connecting $a$ and $b$.
Let $\alpha _{+}$ and $\alpha _{-}$ be the endpoints of $a$, and $\beta _{+}$ and $\beta _{-}$ the endpoints of $b$.
Then, we define

$P(a, b)=
\begin{cases}
\{ (a, b)\} & {\rm if} \ a\cap b=\emptyset,\\
\{{\mathcal P}(a^{\alpha_{+} },b^{\beta_{+} }),{\mathcal P}(a^{\alpha_{+} },b^{\beta_{-}}),{\mathcal P}(a^{\alpha_{-} },b^{\beta_{+}}),{\mathcal P}(a^{\alpha_{-} },b^{\beta_{-} })\} & {\rm if} \ a\cap b\not=\emptyset.
\end{cases}$
\end{defn}

\begin{prop}{\rm(cf.} \cite[Proposition 4.2]{HPW}{\rm)}\label{distances_between_unipaths_and_geodesic}
Let $\mathcal{G}$ be a geodesic in $\mathcal{A}(N)$ between vertices $a$ and $b$.
Then any unicorn arc $c\in \mathcal{P}\in P(a, b)$ is at distance $\leq 6$ from ${\mathcal G}$.
\end{prop}

We use ${\mathbb N}$ for the set of all natural numbers (not including $0$). 

\begin{lem}\label{a}
Let $x_{0}, x_{1},\ldots, x_{m}$ $(m\leq 2^{k}, k\in {\mathbb N})$ be a sequence of vertices of $\mathcal{A}(N)$.
Then for any $\mathcal{P}\in P(a, b)$ and any $c\in \mathcal{P}$, there exist $0\leq i<m$ and $c^{*}\in \mathcal{P}^{*}\in P(x_{i}, x_{i+1})$ such that $d_{\mathcal{A}}(c, c^{*})\leq k$.
\end{lem}

\begin{proof}[Proof of Lemma~\ref{a}]
We prove this by induction of $k$.
Suppose that $k=1$. 
If $m=0$, then $P(x_{0}, x_{0})=\{(x_{0}, x_{0})\}$.
Indeed, $x_{0}$ is an arc and its regular neighborhood is a band, and then there exists an arc which is homotopic to $x_{0}$ and disjoint from $x_{0}$. 
If $m=1$, then we set $x_{0}=a$ and $x_{1}=b$.
By Proposition~\ref{rem1}, for any $\mathcal{P}\in P(a, b)$ and $c_{i}\in \mathcal{P}$, the unicorn arc $c_{i+1}\in \mathcal{P}$ satisfies $d_{\mathcal{A}}(c_{i}, c_{i+1})=1\leq 2$. 
If $m=2$, the we set $x_{0}=a$, $x_{1}=d$, and $x_{2}=b$.
We choose one of the endpoints $\alpha _{+}$, $\beta _{+}$, and $\delta _{+}$ of $a$, $b$, and $d$.
By Lemma~\ref{key1}, for any $c\in \mathcal{P}(a^{\alpha _{+}}, b^{\beta _{+}})$, there exists $c'\in \mathcal{P}(a^{\alpha _{+}}, d^{\delta _{+}})\cup \mathcal{P}(b^{\beta _{+}}, d^{\delta _{+}})$ such that $d_{\mathcal{A}}(c, c')=1\leq 2$.
Hence the case of $k=1$ is done.

Suppose that for all $m\leq 2^{k}$, the statement of Lemma~\ref{a} is satisfied.
For any $2^{k}<m\leq 2^{k+1}$ and any sequence $x_{0}, x_{1}, \ldots, x_{m}$ of vertices of $\mathcal{A}(N)$, set $x_{0}=a$, $x_{2^{k}}=d$, and $x_{m}=b$. By Lemma~\ref{key1}, for any $\mathcal{P}_{1}\in P(a,b)$ and any $c\in \mathcal{P}_{1}$, there exists $c'\in \mathcal{P}_{2}\cup \mathcal{P}_{3}\in P(a, d)\cup P(d, b)$, where $\mathcal{P}_{2}\in P(a, d)$ and $\mathcal{P}_{3}\in P(d, b)$, such that $d_{\mathcal{A}}(c, c')=1$. 
If $c'\in \mathcal{P}_{2}$, then by the assumption of the induction, there exist $0\leq i<2^{k}$ and $c^{*}\in \mathcal{P}^{*}\in P(x_{i}, x_{i+1})$ such that $d_{\mathcal{A}}(c', c^{*})\leq k$.
Thus, we get $d_{\mathcal{A}}(c, c^{*})\leq d_{\mathcal{A}}(c, c')+d_{\mathcal{A}}(c', c^{*}) \leq k+1$.
If $c'\in \mathcal{P}_{3}$, then there also exist $2^{k}\leq i<m$ and $c^{*}\in \mathcal{P}^{*}\in P(x_{i}, x_{i+1})$ such that $d_{\mathcal{A}}(c', c^{*})\leq k$ because the sequence of vertices $x_{2^{k}}, \ldots, x_{m}$ consists of less than or equal to $2^{k}+1$ vertices of $\mathcal{A}(N)$ and because of the hypothesis of the induction. 
Hence, we get $d_{\mathcal{A}}(c, c^{*})\leq d_{\mathcal{A}}(c, c')+d_{\mathcal{A}}(c', c^{*})\leq k+1$, as desired.
\end{proof}

\begin{proof}[Proof of Proposition~\ref{distances_between_unipaths_and_geodesic}]
Fix an arbitrary unicorn path $\mathcal{P}\in P(a, b)$.
Let $c\in \mathcal{P}$ be at maximal distance $k$ from $\mathcal{G}$. Assume that $k\geq 1$. The goal of this proof is to show that $k\leq 6$.
We take the maximal subpath $[a', b']\subset \mathcal{P}$ which fills three conditions $c\in [a', b']$, $d_{\mathcal{A}}(c, a')\leq2k$, and $d_{\mathcal{A}}(c, b')\leq2k$.
Let $\alpha$ and $\beta$ be one of the endpoints of $a$ and $b$.
By Lemma~\ref{key3}, either $\mathcal{P}(a'^{\alpha}, b'^{\beta})$ becomes subpath $[a', b']$ of $\mathcal{P}\in P(a, b)$, or $a'$ and $b'$ represent adjacent vertices of $\mathcal{A}(N)$ and $[a', b']=(a, c, b)$.
First, we consider the latter case.
By the conditions of $[a', b']$, we get $a'=a$ and $b'=b$. We see $\mathcal{G}=(a, b)$, since $\mathcal{G}$ is a geodesic in $\mathcal{A}(N)$ connecting $a$ and $b$, and $a=a'$ and $b=b'$ represent adjacent vertices of $\mathcal{A}(N)$.
Hence, we get $d_{\mathcal{A}}(c, \mathcal{G})=1\leq 6$.
Second, we consider the former case.
Let $a''$ and $b''$ be the closest vertices in $\mathcal{G}$ to $a'$ and $b'$ in $\mathcal{A}(N)$. 
It follows that $d_{\mathcal{A}}(a', a'')\leq k$ and $d_{\mathcal{A}}(b' , b'')\leq k$.

Then,
\begin{align*}
d_{\mathcal{A}}(a'', b'') &\leq d_{\mathcal{A}}(a'', a')+d_{\mathcal{A}}(a', c)+d_{\mathcal{A}}(c, b')+d_{\mathcal{A}}(b', b'')\\
&\leq k+2k+2k+k\\
&= 6k.
\end{align*}

Let $a'a''$, $b'b''$, and $a''b''$ be geodesics in $\mathcal{A}(N)$ connecting $a'$ and $a''$, $b'$ and $b''$, and $a''$ and $b''$. 
Note that $a''b''$ is a subpath of $\mathcal{G}$. 
It follows that 
\[d_{\mathcal{A}}(a', b')\leq d_{\mathcal{A}}(a', a'')+d_{\mathcal{A}}(a'', b'')+d_{\mathcal{A}}(b'', b')\leq k+6k+k=8k.\] 
Suppose that the length of $a'a''\cup b'b''\cup a''b''$ is $m$.
We get $m\leq 8k$. 
Let $\{x_{i}\}_{i=0}^{m}$ be the sequence of the vertices of $a'a''\cup b'b''\cup a''b''$, where $x_{i}$ is adjacent to $x_{i+1}$ for each $i=0,\ldots, m-1$, and $x_{0}=a'$, $x_{m}=b''$.
By Lemma~\ref{a}, for $c\in \mathcal{P}$, there exists $0\leq i<m$ such that $d_{\mathcal{A}}(c, x_{i})\leq\lceil \log_2 8k \rceil$. 
For this $x_{i}$, we claim that $d_{\mathcal{A}}(c, x_{i})\geq k$. 
Indeed, if $x_{i}\in \mathcal{G}$, then $d_{\mathcal{A}}(c, x_{i})\geq d_{\mathcal{A}}(c, \mathcal{G})=k$.
If $x_{i}\not\in \mathcal{G}$ and $x_{i}\in a'a''$, then $a'\not=a''$, and so $d_{\mathcal{A}}(c, a')=2k$. Thus, 
\begin{align*}
d_{\mathcal{A}}(c, x_{i})&\geq  d_{\mathcal{A}}(c, a')- d_{\mathcal{A}}(x_{i}, a')\\
&\geq 2k-k=k.
\end{align*}
If $x_{i}\not\in \mathcal{G}$ and $x_{i}\in b'b''$, then we also get $d_{\mathcal{A}}(c, x_{i})\geq k$.

Therefore, we get $k\leq \lceil \log_2 8k \rceil$, and so $k\leq 6$.
\end{proof}

\begin{proof}[Proof of Theorem~\ref{second_thm}]
Let $T=abd$ be any geodesic triangle in $\mathcal{A}(N)$, where $a$, $b$, and $d$ are three vertices of $\mathcal{A}(N)$.
By Lemma~\ref{key2}, for $a$, $b$, and $d$, there exist $c_{ab}\in\mathcal{P}(a^{\alpha} ,b^{\beta } )$, $c_{bd}\in\mathcal{P}(b^{\beta }, d^{\delta })$, and $c_{da}\in\mathcal{P}(d^{\delta }, a^{\alpha}  )$ such that each pair represents adjacent vertices of $\mathcal{A}(N)$.
Let $ab$, $bd$, and $da$ be three sides of $T$ connecting $a$ and $b$, $b$ and $d$, and $d$ and $a$ in $\mathcal{A}(N)$.
By Proposition~\ref{distances_between_unipaths_and_geodesic}, $c_{ab}$ is at distance $\leq 6$ from $ab$, and $\leq 7$ from both $bd$ and $da$.
Hence, $c_{ab}$ is a $7$-center of $T$ (see Figure~\ref{fig_second_thm}).
\end{proof}

\begin{figure}[h]
\includegraphics[scale=1.0]{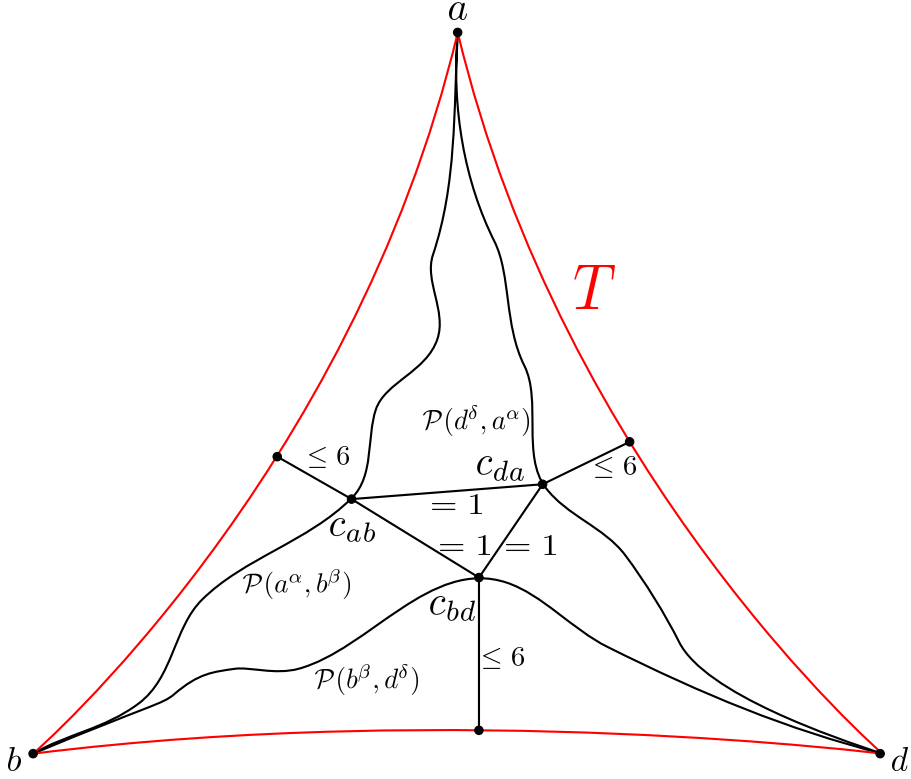}
\caption{Arc graphs are $7$-hyperbolic ($abd$ has a $7$-center $c_{ab}$ in $\mathcal{A}(N)$).}\label{fig_second_thm}
\end{figure}

\section{Curve graphs are uniformly hyperbolic}\label{Curve graphs are hyperbolic}

\begin{prop}{\rm(}\cite[Theorem 6.1]{Szepietowski}{\rm)}\label{Szepietowski}
The complex of curve of $N_{g, n}$ is $(g-3)$-connected if $n=0,1$, and $(g+n-5)$-connected if $n\geq 2$.
\end{prop}

By Proposition~\ref{Szepietowski}, we get the following:

\begin{cor}
If $g=1, 2$ and $g+n\geq 5$, or $g\geq 3$, then the curve graph $\mathcal{C}(N_{g, n})$ is path-connected.
\end{cor}

We define a retraction $r:\mathcal{AC}(N)\rightarrow \mathcal{C}(N)$ as follows.
If $a\in \mathcal{C}^{(0)}(N)$, then $r(a)=a$.
If $a\in \mathcal{A}^{(0)}(N)$, then we assign a boundary component of a regular neighborhood of its union with $\partial N$ to $r(a)$ (see Figure~\ref{fig_retraction}).
Note that if there are two boundary components of the regular neighborhood, then we choose essential one, that is, a curve which does not bound a disk or a M\"{o}bius band and is not homotopic to a boundary component of $N$ (c.f. $r'\colon \mathcal{AC}(S)\rightarrow \mathcal{C}(S)$ in~\cite{HPW}).

The difference from $r'$ in~\cite{HPW} is as follows: if $a$ is an arc on $N$ which goes through crosscaps odd number of times, then $r(a)$ is ``twisted" (see the left-hand side of Figure~\ref{fig_r(a)_is_twisted_and_not_twisted}).

\begin{figure}[h]
\includegraphics[scale=1.3]{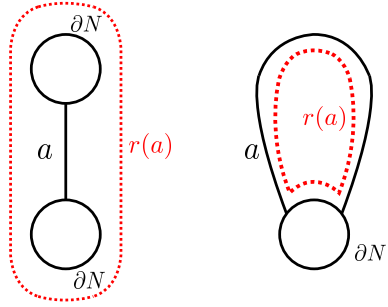}
\caption{Examples of the retraction $r$.}\label{fig_retraction}
\end{figure}

\begin{figure}[h]
\includegraphics[scale=0.275]{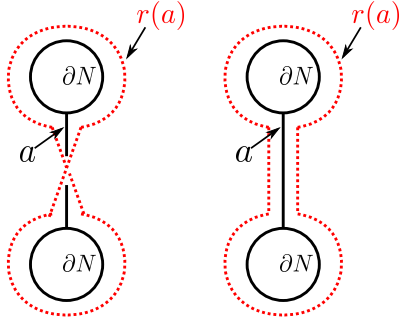}
\caption{Examples that $r(a)$ is twisted (left) and untwisted (right).}\label{fig_r(a)_is_twisted_and_not_twisted}
\end{figure}

\begin{figure}[h]
\includegraphics[scale=0.90]{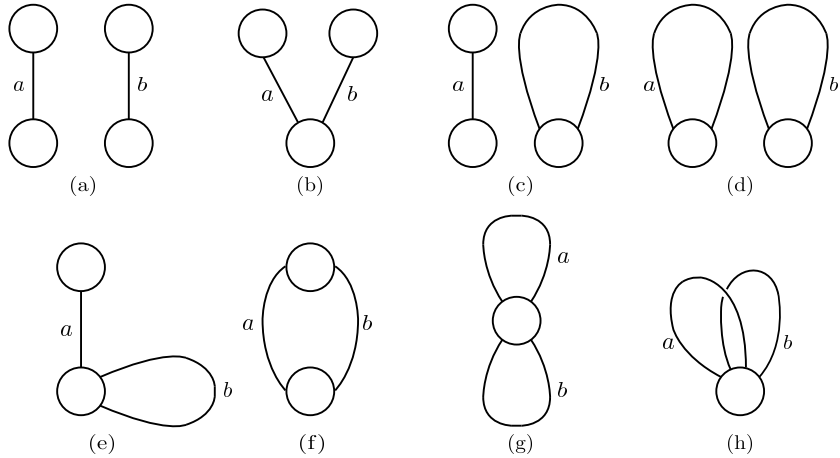}
\caption{Eight cases of $a, b\in \mathcal{A}^{(0)}(N)$ which satisfy $d_{\mathcal{AC}}(a, b)=1$.}\label{fig_2-Lipschitz}
\end{figure}

\newpage
\begin{lem}\label{r_2-Lipschitz}
The retraction $r$ is $2$-Lipschitz, namely, $d_{\mathcal{C}}(r(a), r(b))\leq 2d_{\mathcal{AC}}(a, b)$ for any $a,b\in \mathcal{AC}(N)$.
\end{lem}

\begin{proof}
It is enough to prove that $d_{\mathcal{C}}(r(a), r(b))\leq 2$ for $a, b\in \mathcal{AC}(N)$ with $d_{\mathcal{AC}}(a, b)=1$. 

Case 1: if $a, b\in \mathcal{C}^{(0)}(N)$, then $d_{\mathcal{C}}(r(a), r(b))=d_{\mathcal{C}}(a, b)=d_{\mathcal{AC}}(a, b)=1<2$.

Case 2: if $a\in \mathcal{C}^{(0)}(N)$ and $b\in \mathcal{A}^{(0)}(N)$, then we can take the regular neighborhood of the union of $b$ and the boundary components which have endpoints of $b$ without intersecting $a$.
Note that $r(b)$ may coincide with $a$. Thus $d_{\mathcal{C}}(r(a), r(b))=d_{\mathcal{C}}(a, r(b))\leq 1< 2$.

Case 3: if $a, b\in \mathcal{A}^{(0)}(N)$, then there are eight types of pairs of $a$, $b$ which fill $d_{\mathcal{AC}}(a, b)=1$ (see Figure~\ref{fig_2-Lipschitz}, where each circle represents a boundary component of $N$).
Note that there are two cases where $a$ (resp. $b$) passes through crosscaps odd number of times, and where it  passes through crosscaps even number of times.
In the former case, we say that $r(a)$ (resp. $r(b)$) is {\it twisted} (see the left-hand side of Figure~\ref{fig_r(a)_is_twisted_and_not_twisted}), and in the latter case, we say that $r(a)$ (resp. $r(b)$) is {\it untwisted} (see the right-hand side of Figure~\ref{fig_r(a)_is_twisted_and_not_twisted}).

(1) The case $(g, n)\not=(3, 1)$

In the case of (a) in Figure~\ref{fig_2-Lipschitz}, $r(a)$ and $r(b)$ become essential disjoint curves. Since the genus of $N$ is at least $1$, we get $r(a)\not=r(b)$. Hence, $d_{\mathcal{C}}(r(a), r(b))=1<2$.

In the case of (b) in Figure~\ref{fig_2-Lipschitz},  there are three cases where both $r(a)$ and $r(b)$ are untwisted, $r(a)$ is untwisted and $r(b)$ is twisted, and both $r(a)$ and $r(b)$ are twisted.
In all three cases, we take a boundary component $\alpha$ of a regular neighborhood of the union of $a$ and $b$ with $\partial N$ large enough to intersect neither $r(a)$ nor $r(b)$. 
Then it is sufficient to prove that $\alpha$ is essential. 
It is clear that $\alpha$ bounds $3$ punctured disk on one side. 
We show that $\alpha$ does not bound a disk, an annulus, or a M\"{o}bius band on the other side. 
By a calculation of the Euler characteristics, we see that $\alpha$ separates $N$ into $N_{0, 4}$ and $N_{g, n-2}$. 
If $g\geq 2$, then $N_{g, n-2}$ is not a disk, an annulus, or a M\"{o}bius band. If $g=1$, then $N_{g, n-2}$ is also not a disk, an annulus, or a M\"{o}bius band, since $g+n\geq 5$. 
Therefore, $\alpha$ is essential and $d_{\mathcal{C}}(r(a), r(b))\leq d_{\mathcal{C}}(r(a), \alpha)+d_{\mathcal{C}}(\alpha, r(b))\leq 2$.

In the case of (c) and (d) in Figure~\ref{fig_2-Lipschitz}, $r(a)$ and $r(b)$ are essential and disjoint curves. 
Note that $r(a)$ and $r(b)$ may coincide. 
Hence, $d_{\mathcal{C}}(r(a), r(b))\leq 1<2$.

In the case of (e) in Figure~\ref{fig_2-Lipschitz}, there are four cases where both $r(a)$ and $r(b)$ are untwisted, $r(a)$ is untwisted and $r(b)$ is twisted, $r(a)$ is twisted and $r(b)$ is untwisted, and both $r(a)$ and $r(b)$ are twisted. 
Let $\gamma _{1}$ and $\gamma _{2}$ be the boundary components of $N$ which have endpoints of $a$ and $b$.
In the first case, i.e. both $r(a)$ and $r(b)$ are untwisted, there are two boundary components of the regular neighborhood of $a\cup \gamma _{1}\cup \gamma _{2}\cup b$. 
We denote by $\alpha$ the outer part of the regular neighborhood, and by $\alpha'$ the other (see Figure~\ref{fig_2-Lipschitz_e}). 
Note that $\alpha$ and $\alpha'$ intersect neither $r(a)$ nor $r(b)$. 
It is sufficient to show that at least one of $\alpha$ and $\alpha'$ is essential. 
If $\alpha$ bounds a disk, then we take $\alpha'$. 
The curve $\alpha'$ separates $N$ into $N_{0, 3}$ and $N_{g, n-1}$. 
If $g\geq 2$, then $N_{g, n-1}$ is not a disk, an annulus, or a M\"{o}bius band. 
If $g=1$, then $N_{g, n-1}$ is also not a disk, an annulus, or a M\"{o}bius band, for $g+n\geq 5$. Hence, $\alpha'$ is essential. 
If $\alpha$ bounds an annulus or a M\"{o}bius band, then we take $\alpha'$. 
The curve $\alpha'$ separates $N$ into $N_{0, 4}$ and $N_{g, n-2}$, or $N_{1, 3}$ and $N_{g-1, n-1}$, respectively. 
By a similar argument to that of (b), $N_{g, n-2}$ is not a disk, an annulus, or a M\"{o}bius band.
We consider $N_{g-1, n-1}$.
If $g-1\geq 2$, then $N_{g-1, n-1}$ is not a disk, an annulus, or a M\"{o}bius band. If $g-1=1$ or $0$, then $N_{g-1, n-1}$ is not a disk, an annulus, or a M\"{o}bius band, for $g+n\geq 5$. 
Hence $\alpha'$ is essential.
If $\alpha$ does not bound a disk, an annulus, or a M\"{o}bius band, then we take $\alpha$, and so $\alpha$ is essential.
In the second case, i.e. $r(a)$ is untwisted and $r(b)$ is twisted, we can show it by a similar argument to that of the first case in (e).
In the third case, i.e. $r(a)$ is twisted and $r(b)$ is untwisted, there is one boundary component of the regular neighborhood of $a\cup \gamma _{1}\cup \gamma _{2}\cup b$, and we denote it by $\alpha$. 
It is sufficient to show that $\alpha$ is essential. 
The curve $\alpha$ separates $N$ into $N_{1, 3}$ and $N_{g-1, n-1}$. 
By a similar argument to that of the first case in (e), $\alpha$ is essential.
In the last case, i.e. both $r(a)$ and $r(b)$ are twisted, we can show it by a similar argument to that of the third case in (e).

\begin{figure}[h]
\includegraphics[scale=0.5]{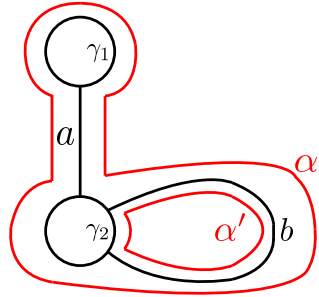}
\caption{The case where both $r(a)$ and $r(b)$ are untwisted in (e).}\label{fig_2-Lipschitz_e}
\end{figure}

In the case of (f) in Figure~\ref{fig_2-Lipschitz}, there are three cases where both $r(a)$ and $r(b)$ are untwisted, $r(a)$ is untwisted and $r(b)$ is twisted, and both $r(a)$ and $r(b)$ are twisted.
Let $\gamma _{1}$ and $\gamma _{2}$ be the boundary components of $N$ which have endpoints of $a$ and $b$.
In the first case, i.e. both $r(a)$ and $r(b)$ are untwisted, there are two boundary components of the regular neighborhood of $a\cup \gamma _{1}\cup \gamma _{2}\cup b$. We denote by $\alpha$ the outer part of the regular neighborhood, and by $\alpha'$ the other (see Figure~\ref{fig_2-Lipschitz_f}). 
If $\alpha$ bounds a disk, an annulus, or a M\"{o}bius band, we take $\alpha'$. 
The curve $\alpha'$ separates $N$ into $N_{0, 3}$ and $N_{g, n-1}$, $N_{0, 4}$ and $N_{g, n-2}$, or $N_{1, 3}$ and $N_{g-1, n-1}$ respectively, and so $\alpha$ is essential.
If $\alpha$ does not bound a disk, an annulus, or a M\"{o}bius band, then we take $\alpha$, which is essential.
In the second case, i.e. $r(a)$ is untwisted and $r(b)$ is twisted, there is one boundary component of the regular neighborhood of $a\cup \gamma _{1}\cup \gamma _{2}\cup b$, and we denote it by $\alpha$. 
The curve $\alpha$ separates $N$ into $N_{1, 3}$ and $N_{g-1, n-1}$, and so $\alpha$ is essential.
In the last case, i.e. both $r(a)$ and $r(b)$ are twisted, there are two boundary components of the regular neighborhood of $a\cup \gamma _{1}\cup \gamma _{2}\cup b$. 
We take one of them and denote it by $\alpha$. 
Then $\alpha$ is a non-separating curve on $N$. 
Therefore, $\alpha$ is essential.

\begin{figure}[h]
\includegraphics[scale=0.4]{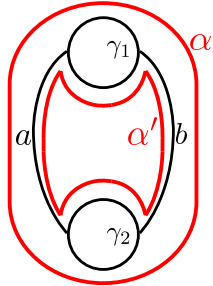}
\caption{The case where both $r(a)$ and $r(b)$ are untwisted in (f).}\label{fig_2-Lipschitz_f}
\end{figure}

In the case of (g) in Figure~\ref{fig_2-Lipschitz}, there are three cases where both $r(a)$ and $r(b)$ are untwisted, $r(a)$ is untwisted and $r(b)$ is twisted, and both $r(a)$ and $r(b)$ are twisted.
Let $\gamma $ be a boundary component of $N$ which has endpoints of $a$ and $b$.
In the first case, i.e. both $r(a)$ and $r(b)$ are untwisted, there are three boundary components of the regular neighborhood of $a\cup \gamma\cup b$.
We denote by $\alpha_{1}$ the component which encloses $a$, $\gamma$, and $b$, and by $\alpha_{2}$ (resp. $\alpha_{3}$) the component which lies the inner part of $a$ (resp. $b$) in Figure~\ref{fig_2-Lipschitz_g}.
Suppose that $\alpha_{1}$ bounds a disk. 
It is sufficient to show that $\alpha_{3}$ is essential if $\alpha_{2}$ is not essential. (If $\alpha_{2}$ is essential, then we take $\alpha_{2}$.) 
When we assume that $\alpha_{2}$ is not essential, $\alpha_{2}$ bounds either an annulus or a M\"{o}bius band. 
Then, the curve $\alpha_{3}$ separates $N$ into either $N_{0, 3}$ and $N_{g, n-1}$, or $N_{1, 2}$ and $N_{g-1, n}$.
We can show that  $N_{g-1, n}$ is also not a disk, an annulus, or a M\"{o}bius band, and so $\alpha_{3}$ is essential.
When $\alpha_{1}$ bounds an annulus and $\alpha_{2}$ is not essential, we can also take an essential curve $\alpha_{3}$ which is disjoint from both $r(a)$ and $r(b)$. 
Suppose that $\alpha_{1}$ bounds a M\"{o}bius band and $\alpha_{2}$ is not essential. 
Then, $\alpha_{2}$ bounds either an annulus or a M\"{o}bius band, and so the curve $\alpha_{3}$ separates $N$ into either $N_{1, 3}$ and $N_{g-1, n-1}$, or $N_{2, 2}$ and $N_{g-2, n}$.  
By a similar argument to that of third case in (e), $N_{g-1, n-1}$ is not a disk, an annulus, or a M\"{o}bius band. 
We consider $N_{g-2, n}$.
If $g-2\geq 2$, then $N_{g-2, n}$ is not a disk, an annulus, or a M\"{o}bius band. 
If $g-2=1$, then $N_{g-2, n}$ is also not a disk, an annulus, or a M\"{o}bius band because of the assumption of $(g, n)\not=(3, 1)$. 
If $g-2=0$, then $N_{g-2, n}$ is also not a disk, an annulus, or a M\"{o}bius band, since $g+n\geq 5$. 
When $\alpha_{1}$ does not bound a disk, an annulus, or a M\"{o}bius band, we take $\alpha_{1}$. 

In the second case, i.e. $r(a)$ is untwisted and $r(b)$ is twisted, there are two boundary components of the regular neighborhood of $a\cup \gamma \cup b$, and  the regular neighborhood of $a\cup \gamma \cup b$ is a non-orientable surface of genus $1$ with $3$ boundary components. 
We denote by $\alpha_{1}$ and $\alpha_{2}$ the boundaries of this surface which are not $\gamma$. 
It is sufficient to show that, if $\alpha_{1}$ is not essential, then $\alpha_{2}$ is essential. 
If $\alpha_{1}$ bounds a disk, an annulus, or a M\"{o}bius band, then $\alpha_{2}$ bounds $N_{1,2}$ and $N_{g-1, n}$, $N_{1, 3}$ and $N_{g-1, n-1}$, or $N_{2, 2}$ and $N_{g-2, n}$. We get $\alpha$ is essential.
In the third case, i.e. both $r(a)$ and $r(b)$ are twisted, there is one boundary component of the regular neighborhood of $a\cup \gamma \cup b$ (we denote it by $\alpha$), and the regular neighborhood of $a\cup \gamma \cup b$ is a non-orientable surface of genus $2$ with $2$ boundary components. 
Then $\alpha$ bounds $N_{2, 2}$ and $N_{g-2, n}$, and so $\alpha$ is essential.

\begin{figure}[h]
\includegraphics[scale=1.3]{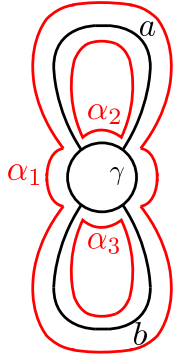}
\caption{The case where both $r(a)$ and $r(b)$ are untwisted in (g).}\label{fig_2-Lipschitz_g}
\end{figure}

In the case of (h) in Figure~\ref{fig_2-Lipschitz}, there are three cases where both $r(a)$ and $r(b)$ are untwisted, $r(a)$ is untwisted and $r(b)$ is twisted, and both $r(a)$ and $r(b)$ are twisted.
Let $\gamma $ be a boundary component of $N$ which has endpoints of $a$ and $b$.
In the first case, i.e. both $r(a)$ and $r(b)$ are untwisted, the regular neighborhood of $a\cup \gamma \cup b$ is twice hold torus, and $r(a)$ and $r(b)$ intersect once. Hence, the complement of $r(a)$ and $r(b)$ is a twice hold disk, and then we can take an essential curve which goes around the twice hold disk.
In the second case, i.e. $r(a)$ is untwisted and $r(b)$ is twisted, it is enough to give the same argument as we gave in the third case of (g).
In the third case, i.e. both $r(a)$ and $r(b)$ are twisted, it is enough to give the same argument as we gave in the second case of (g). 

In the cases of (e), (f), (g), and (h), there is an essential curve $\alpha$ which is intersect neither $r(a)$ nor $r(b)$. 
Therefore, $d_{\mathcal{C}}(r(a), r(b))\leq d_{\mathcal{C}}(r(a), \alpha)+d_{\mathcal{C}}(\alpha, r(b))\leq 2$.
 
(2) The case $(g, n)=(3, 1)$ 

By the argument mentioned above, it is enough to discuss only the case of (g).

If both $r(a)$ and $r(b)$ are untwisted and $\alpha_{1}$ bounds a M\"{o}bius band, then $\alpha_{2}$ bounds a M\"{o}bius band and $\alpha_{3}$ also bounds a M\"{o}bius band, since $(g, n)=(3, 1)$. We take a curve which passes through a M\"{o}bius band, and this curve is essential and intersects neither $r(a)$ nor $r(b)$.
If $r(a)$ is untwisted, $r(b)$ is twisted, and $\alpha_{1}$ bounds a M\"{o}bius band, then $\alpha_{2}$ bounds $N_{2, 1}$ and a M\"{o}bius band, since $(g, n)=(3, 1)$. We take the curve which passes through the M\"{o}bius band in the exterior of the regular neighborhood of $a\cup \gamma \cup b$.
If both $r(a)$ and $r(b)$ are twisted and $\alpha_{1}$ bounds a M\"{o}bius band, then $\alpha_{1}$ also bounds $N_{1, 1}$, since $(g, n)=(3, 1)$. We take the curve which passes through the M\"{o}bius band in the exterior of the regular neighborhood of $a\cup \gamma \cup b$.
In (2), there is an essential curve $\alpha$ which is intersect neither $r(a)$ nor $r(b)$. 
Therefore, $d_{\mathcal{C}}(r(a), r(b))\leq d_{\mathcal{C}}(r(a), \alpha)+d_{\mathcal{C}}(\alpha, r(b))\leq 2$.

(1) and (2) imply that $r$ is a $2$-Lipschitz retraction if $a, b\in \mathcal{A}^{(0)}(N)$, and we complete the proof of Lemma~\ref{r_2-Lipschitz}.
\end{proof}

Before proving Theorem~\ref{mainthm}, we need to show the following proposition.

\begin{prop}
 If $g=1, 2$ and $g+n\geq 5$, or $g\geq 3$, then $\mathcal{AC}(N)$ is connected.
\end{prop}

\begin{proof}
If $a, b\in \mathcal{C}^{(0)}(N)$, then there exists an edge-path connecting $a$ and $b$ in $\mathcal{C}(N)$ from the assumption that $g=1, 2$ and $g+n\geq 5$ or $g\geq 3$.
We consider it as an edge-path in $\mathcal{AC}(N)$.
If $a, b\in \mathcal{A}^{(0)}(N)$, then we connect $a$ and $b$ by a unicorn path in $\mathcal{A}(N)$, and consider it as an edge-path in $\mathcal{AC}(N)$.
Therefore, it is enough to consider the case where $a\in \mathcal{C}^{(0)}(N)$ and $b\in \mathcal{A}^{(0)}(N)$.

Fix any $a\in \mathcal{C}^{(0)}(N)$.
We take an appropriate boundary component $a'$ of the regular neighborhood of $a$, and we connect $a'$ and a boundary component of $N$ by an arc $\eta$ which does not intersect $a$.
Then the products $\eta*a'*\eta ^{-1}$ is a properly embedded arc which is disjoint from $a$.
Hence, we can connect the vertices $a$ and $\eta*a'*\eta ^{-1}$ by an edge in $\mathcal{AC}(N)$.
On the other hand, for any $b\in \mathcal{A}^{(0)}(N)$, we connect it to $\eta*a'*\eta ^{-1}$ in $\mathcal{A}(N)$ by a unicorn path in $P(\eta*a'*\eta ^{-1}, b)$.
Therefore, we can connect an arbitrary $a\in \mathcal{C}^{(0)}(N)$ and an arbitrary $b\in \mathcal{A}^{(0)}(N)$ by an edge-path in $\mathcal{AC}(N)$.
\end{proof}

Now, we give a proof of Theorem~\ref{mainthm}.

\begin{proof}[Proof of Theorem~\ref{mainthm}]
First we assume that $\partial N\not=\emptyset$. 
We take any geodesic triangle $T=abd$ in $\mathcal{C}^{(0)}(N)$, where $a, b, d\in \mathcal{C}^{(0)}(N)$. 
Let $\bar{a}$, $\bar{b}$, and $\bar{d}$ $\in \mathcal{A}^{(0)}(N)$ be three arcs which are adjacent to $a$, $b$, and $d$ in $\mathcal{AC}(N)$ respectively.
we choose one of the endpoints $\alpha$, $\beta$, and $\delta$ of $\bar{a}$, $\bar{b}$, and $\bar{d}$.
Now we prove the following proposition.

\vspace{0.1in}

\begin{prop}~\label{k8}
Let $a$, $b$ be vertices of $\mathcal{C}(N)$, and $\bar{a}$, $\bar{b}$ vertices of $\mathcal{A}(N)$ which are adjacent to $a$, $b$, respectively.
Let $\mathcal{G}=ab$ be a geodesic connecting $a$ and $b$ in $\mathcal{C}(N)$.
Then, any unicorn arc $\bar{c}\in \mathcal{P}\in P(\bar{a}, \bar{b})$ is at distance $\leq 8$ from $\mathcal{G}$.
\end{prop}

\begin{proof}
Fix an arbitrary unicorn path $\mathcal{P}\in P(\bar{a}, \bar{b})$.
Let $\bar{c}\in \mathcal{P}$ be at maximal distance $k$ from $\mathcal{G}$.
Assume that $k\geq 1$. Similarly to the proof of Proposition~\ref{distances_between_unipaths_and_geodesic}, we take the maximal subpath $[\bar{a}', \bar{b}']\in \mathcal{P}$ which fills the conditions $\bar{c}\in [\bar{a}', \bar{b}']$, $d_{\mathcal{AC}}(\bar{a}', \bar{c})\leq 2k$, and $d_{\mathcal{AC}}(\bar{b}', \bar{c})\leq 2k$. Let $a''$ and $b''$ be the closest vertices of $\mathcal{G}$ to $\bar{a}'$ and $\bar{b}'$ in $\mathcal{AC}(N)$. 

Then, 
\begin{align*}
d_{\mathcal{AC}}(a'', b'') &\leq d_{\mathcal{AC}}(a'', \bar{a}')+d_{\mathcal{AC}}(\bar{a}', \bar{c})+d_{\mathcal{AC}}(\bar{c}, \bar{b}')+d_{\mathcal{AC}}(\bar{b}', b'')\\
&\leq k+2k+2k+k\\
&= 6k.
\end{align*}

Let $[a'', b'']_{\mathcal{G}}$ be the subpath of $\mathcal{G}$ connecting $a''$ and $b''$.
By the inequality $d_{\mathcal{AC}}(a'', b'')\leq 6$ and Lemma~\ref{r_2-Lipschitz}, the length of $[a'', b'']_{\mathcal{G}}$ in $\mathcal{C}(N)$ is at most $12k$.
Let $\bar{a}'a''$ and $b''\bar{b}'$ be geodesics in $\mathcal{AC}(N)$ connecting $\bar{a}'$ and $a''$, and $b''$ and $\bar{b}'$.
Then $d_{\mathcal{AC}}(\bar{a}', a'')\leq k$ and $d_{\mathcal{AC}}(\bar{b}', b'')\leq k$.
The length of  $\bar{a}'a''\cup [a'', b'']_{\mathcal{G}}\cup b''\bar{b}'$ is at most $14k$. 
Let $\{x_{i}\}_{i=0}^{m}$ be the sequence of the vertices of $\bar{a}'a''\cup [a'', b'']_{\mathcal{G}}\cup b''\bar{b}'$, where $x_{i}$ is adjacent to $x_{i+1}$ for each $i=0,\ldots, m-1$, and $x_{0}=\bar{a}'$, $x_{m}=\bar{b}'$.
Then, we get $m\leq 14k$.
Furthermore, for any $i=1,\ldots, m-1$, there exists a vertex $\overline{x_{i}}\in \mathcal{A}^{(0)}(N)$ which is adjacent to both $x_{i}$ and $x_{i+1}$ or equal to either $x_{i}$ or $x_{i+1}$.
We set $\overline{x_{0}}=\bar{a}'$ and $\overline{x_{m}}=\bar{b}'$.
Then $\{ \overline{x_{i}} \}_{i=0}^{m}$ is a sequence of vertices of $\mathcal{A}(N)$, where $m\leq 14k$.
By Lemma~\ref{a}, for $\bar{c}\in \mathcal{P}(\bar{a}'^{\alpha}, \bar{b}'^{\beta})$, there exist $0\leq i<m$ and $c^{*}\in \mathcal{P}^{*}\in P(\overline{x_{i}}, \overline{x_{i+1}})$ such that $d_{\mathcal{AC}}(\bar{c}, c^{*})\leq \lceil \log_2 14k \rceil$.
Note that all unicorn arcs of unicorn paths in \\
$P(\overline{x_{i}}, \overline{x_{i+1}})$ are disjoint from $x_{i+1}$. 

Then, we get 
\begin{align*}
d_{\mathcal{AC}}(\bar{c}, x_{i+1}) &\leq d_{\mathcal{AC}}(\bar{c}, c^{*})+d_{\mathcal{AC}}(c^{*}, x_{i+1})\\
&\leq d_{\mathcal{A}}(\bar{c}, c^{*})+d_{\mathcal{AC}}(c^{*}, x_{i+1})\\
&\leq \lceil \log_2 14k \rceil+1.
\end{align*}

For this $x_{i+1}\in \mathcal{AC}^{(0)}(N)$, we claim that $d_{\mathcal{AC}}(\bar{c}, x_{i+1})\geq k$.
Indeed, if $x_{i+1}\in [a'', b'']_{\mathcal{G}}\subset \mathcal{G}$, then $d_{\mathcal{AC}}(\bar{c}, x_{i+1})\geq d_{\mathcal{AC}}(\bar{c}, \mathcal{G})=k$. 
If $x_{i+1}\not\in [a'', b'']_{\mathcal{G}}$ and $x_{i+1}\in \bar{a}'a''$, then $d_{\mathcal{AC}}(\bar{c}, \bar{a}')=2k$, since $\bar{a}'\not=\bar{a}$.
Thus 
\begin{align*}
d_{\mathcal{AC}}(\bar{c}, x_{i+1})&\geq  d_{\mathcal{AC}}(\bar{c}, \bar{a}')-d_{\mathcal{AC}}(\bar{a}', \bar{c})\\
&\geq 2k-k=k.
\end{align*}
If $x_{i+1}\not\in [a'', b'']_{\mathcal{G}}$ and $x_{i+1}\in b''\bar{b}'$, then we also get $d_{\mathcal{AC}}(c, x_{i})\geq k$.

Therefore, we get $k\leq \lceil \log_2 14k \rceil+1$, and so $k\leq 8$.
\end{proof}

Now, we go back to the proof of Theorem~\ref{mainthm}.
Let $ab$ be the side of $T$ connecting $a$ and $b$ in $\mathcal{C}(N)$.
From Lemma~\ref{key2}, there exist $c_{\bar{a}\bar{b}}\in\mathcal{P}(\bar{a}^{\alpha}, \bar{b}^{\beta })$, $c_{\bar{b}\bar{d}}\in\mathcal{P}(\bar{b}^{\beta }, \bar{d}^{\delta })$, and $c_{\bar{d}\bar{a}}\in\mathcal{P}(\bar{d}^{\delta }, \bar{a}^{\alpha})$ such that each pair represents adjacent vertices of $\mathcal{A}(N)$. 
By Proposition~\ref{k8}, the vertex $c_{\bar{a}\bar{b}}$ of $\mathcal{AC}(N)$ is a $9$-center of T.
In particular, $c_{\bar{a}\bar{b}}$ is at distance $\leq 8$ from a vertex of $\mathcal{G}=ab$, which is a curve (see Figure~\ref{fig_mainthm}). 
We connect this vertex with $c_{\bar{a}\bar{b}}$ by a geodesic in $\mathcal{AC}(N)$ and call the intermediate vertices $c^{i}$.
Now, we assume that the worst case, that is, the case where there are eight of them. 
We consider $r(c_{\bar{a}\bar{b}}), r(c^{1}),\ldots, r(c^{8})$, where $r$ is the retraction defined at the beginning of Section~\ref{Curve graphs are hyperbolic}.
By Lemma~\ref{r_2-Lipschitz}, the distance between $r(c_{\bar{a}\bar{b}})$ and $r(c^{1})$ is at most $2$, and the distance between $r(c^{i})$ and $r(c^{i+1})$ for each $i=1,\ldots, 6$ is also at most $2$.
Since the vertices on $\mathcal{G}$ are curves, $r(c^{7})$ is adjacent to $r(c^8)$ (see Cases 1 and 2 in the proof of Lemma~\ref{r_2-Lipschitz}).
Hence, $d_{\mathcal{C}}(r(c_{\bar{a}\bar{b}}), r(c^{8}))\leq 15$.
By a similar argument to the one for $\mathcal{G}=ab$ and $c_{\bar{a}\bar{b}}$, we get  $d_{\mathcal{C}}(r(c_{\bar{b}\bar{d}}), bd)\leq 15$ and  $d_{\mathcal{C}}(r(c_{\bar{d}\bar{a}}), da)\leq 15$, where $bd$ and $da$ are the sides of $T$ connecting $b$ and $d$, and $d$ and $a$.
Since $d_{\mathcal{C}}(r(c_{\bar{a}\bar{b}}), r(c_{\bar{b}\bar{d}}))\leq 2$ and $d_{\mathcal{C}}(r(c_{\bar{a}\bar{b}}), r(c_{\bar{d}\bar{a}}))\leq 2$,  the vertex $r(c_{\bar{a}\bar{b}})\in \mathcal{C}^{(0)}(N)$ becomes a $17$-center of the triangle $T$ in $\mathcal{C}(N)$.

\begin{figure}[h]
\includegraphics[scale=1.0]{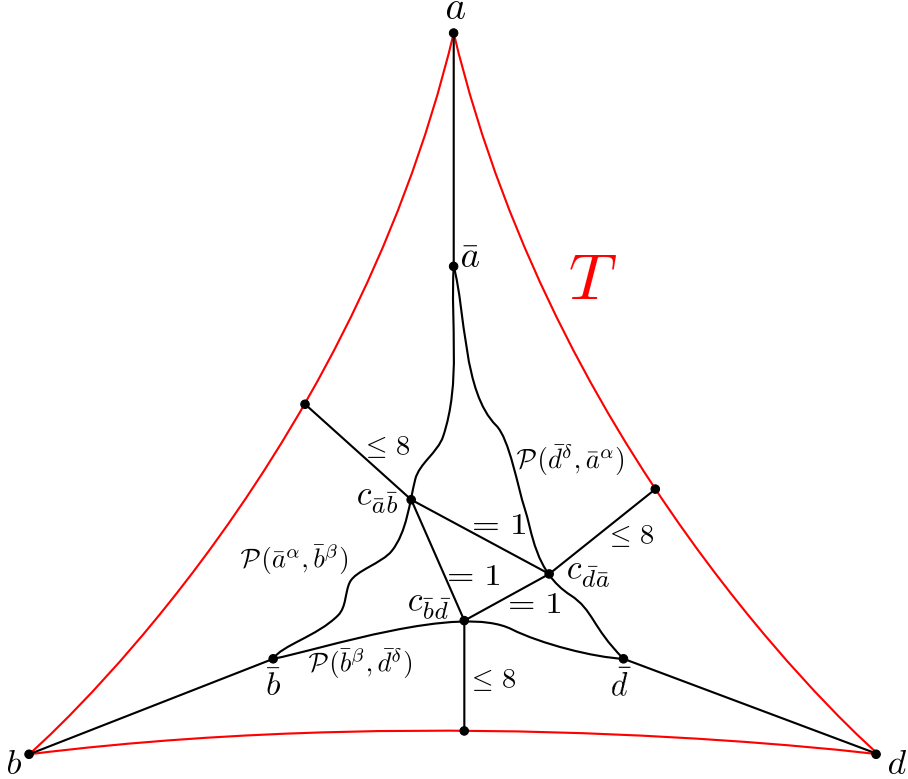}
\caption{$abd$ has a 9-center $c_{\bar{a}\bar{b}}$ in $\mathcal{AC}(N)$.}\label{fig_mainthm}
\end{figure}

Secondly, we assume that $\partial N=\emptyset$.
Note that $N$ has a negative Euler characteristic, since the genus of $N$ is at least $3$.
Let $\overline{N}$ be a surface obtained from $N$ by removing an open disk.
In this proof, we denote by $d_{\mathcal{C}(N)}(\cdot, \cdot)$ and $d_{\mathcal{C}(\overline{N})}(\cdot, \cdot)$ the distances in $\mathcal{C}(N)$ and $\mathcal{C}(\overline{N})$. 
We define a retraction Ret$\colon \mathcal{C}(\overline{N})\rightarrow \mathcal{C}(N)$ as follows: for any $\alpha\in \mathcal{C}(\overline{N})$, Ret$(\alpha)$ is a homotopy class of $\alpha$ in $\mathcal{C}(N)$.
Then Ret is $1$-Lipschitz.
We also define a section Sec$\colon \mathcal{C}(N)\rightarrow \mathcal{C}(\overline{N})$ as follows. 
Choose a hyperbolic metric on $N$.
For any $\alpha\in \mathcal{C}(N)$, we take a geodesic (now we call it $\alpha$) as the representative of $\alpha$.
Then remove $p\in N\setminus \bigcup _{\lambda \in \Lambda } c_{\lambda }$, where each $c_{\lambda }$ is a geodesic on $N$, identify $N- \{p \}$ with $\overline{N}$, and consider Sec$(\alpha)$ as $\alpha$ on $\overline{N}$.
Note that the composition Ret$\circ $Sec is identity on $\mathcal{C}(N)$.

Let $T=abd$ be any geodesic triangle in $\mathcal{C}(N)$, where $a$, $b$, and $d$ are vertices of $\mathcal{C}(N)$.
Since Sec is an embedding map, Sec$(T)=T$ has a $17$-center $q\in \mathcal{C}^{(0)}(\overline{N})$ in $\mathcal{C}(\overline{N})$. 
Let $ab$, $bd$, and $da$ be the sides of $T$ connecting $a$ and $b$, $b$ and $d$, and $d$ and $a$ in $\mathcal{C}(N)$.
Then, for $ab$, we get
\begin{align*}
d_{\mathcal{C}(N)}({\rm Ret}(q), ab) &= d_{\mathcal{C}(N)}{\big(}{\rm Ret}(q), ({\rm Ret}\circ {\rm Sec}(a))({\rm Ret}\circ {\rm Sec}(b)){\big)}\\
&\leq d_{\mathcal{C}(\overline{N})}(q, {\rm Sec}(a){\rm Sec}(b))\\
&\leq 17,
\end{align*}\\
where (Ret$\circ $Sec($a$))(Ret$\circ $Sec($b$)) is a geodesic in $\mathcal{C}(N)$ connecting Ret$\circ $Sec($a$) and Ret$\circ $Sec($b$), and Sec($a$)Sec($b$) is a geodesic in $\mathcal{C}(\overline{N})$ connecting Sec($a$) and Sec($b$).
For $bd$ and $da$, we can show the same results that we showed for $ab$.
Hence, Ret$(q)$ is a $17$-center of $T$ in $\mathcal{C}(N)$.
\end{proof}

\section{Arc-curve graphs are uniformly hyperbolic}\label{Arc-curve graphs are hyperbolic}

\begin{figure}[h]
\includegraphics[scale=1.0]{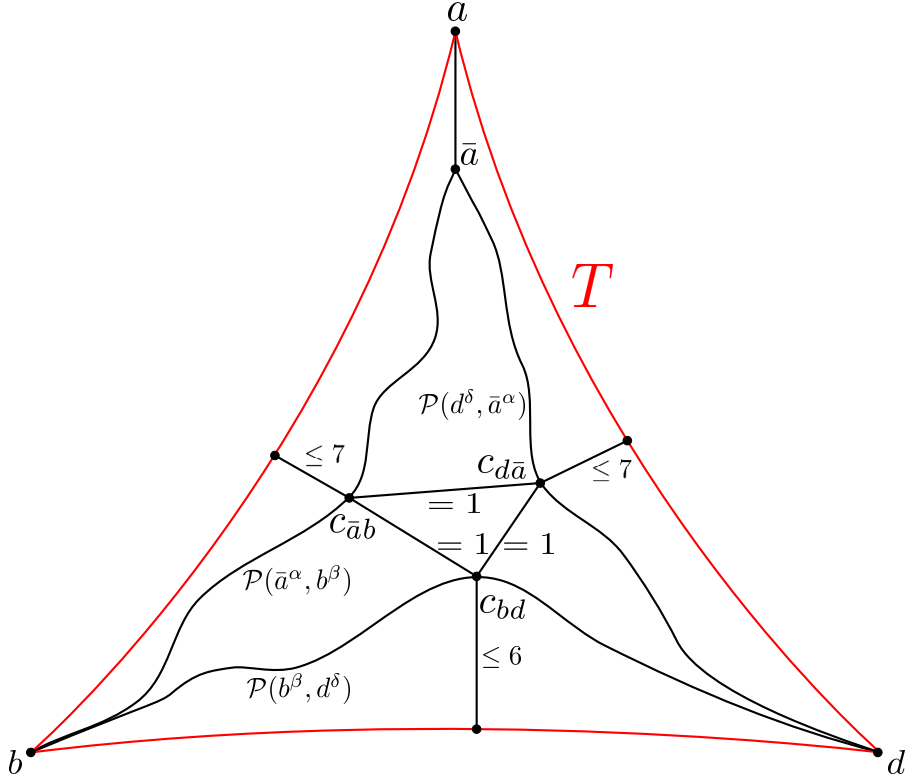}
\caption{$abd$ has an $8$-center $c_{\bar{a}b}$ in $\mathcal{AC}(N)$ ($a\in \mathcal{C}^{(0)}(N)$, $b, d\in \mathcal{A}^{(0)}(N)$).}\label{fig_third_thm}
\end{figure}

\begin{proof}[Proof of Theorem~\ref{third_thm}]
Fix any geodesic triangle $T=abd$ in $\mathcal{AC}(N)$, where $a$, $b$, and $d$ are vertices of $\mathcal{AC}(N)$.

If $a, b, d\in \mathcal{A}^{(0)}(N)$, then $T$ has a $7$-center in $\mathcal{A}(N)$ by Theorem~\ref{second_thm}.
Hence $T$ has a $7$-center in $\mathcal{AC}(N)$.

If $a, b, d\in \mathcal{C}^{(0)}(N)$, then $T$ has a $9$-center in $\mathcal{AC}(N)$ by the proof of Theorem~\ref{mainthm}.

If $a\in \mathcal{C}^{(0)}(N)$ and $b, d\in \mathcal{A}^{(0)}(N)$, then we take $\bar{a}\in \mathcal{A}^{(0)}(N)$ which is adjacent to $a$ in $\mathcal{AC}(N)$.
Similarly to Proposition~\ref{k8}, we can prove the following proposition.
 
\vspace{0.1in}
\begin{prop}~\label{prop_in_third_thm}
Let $a$ be a vertex of $\mathcal{C}(N)$, $b$ a vertex of $\mathcal{A}(N)$, and $\bar{a}$ a vertex of $\mathcal{A}(N)$ which is adjacent to $a$ in $\mathcal{AC}(N)$.
Let $\mathcal{G}=ab$ be a geodesic connecting $a$ and $b$ in $\mathcal{AC}(N)$.
Then, any unicorn arc $\bar{c}\in \mathcal{P}\in P(\bar{a}, b)$ is at distance $\leq 7$ from $\mathcal{G}$.
\end{prop}

By Lemma~\ref{key2}, for $\bar{a}, b, d\in \mathcal{A}^{(0)}(N)$, there exist $c_{\bar{a}b}\in\mathcal{P}(\bar{a}^{\alpha}, b^{\beta} )$,  $c_{bd}\in\mathcal{P}(b^{\beta}, d^{\delta})$, and $c_{d\bar{a}}\in\mathcal{P}(d^{\delta }, \bar{a}^{\alpha} )$ such that each pair represents adjacent vertices of $\mathcal{A}(N)$.
Then, $c_{\bar{a}b}$ is an $8$-center of $T$ in $\mathcal{AC}(N)$ by Proposition~\ref{prop_in_third_thm} (see Figure~\ref{fig_third_thm}).

If $a, b\in \mathcal{C}^{(0)}(N)$ and $d\in \mathcal{A}^{(0)}(N)$, then $T$ has an $8$-center in $\mathcal{AC}(N)$ by a similar argument to that of the case where $a\in \mathcal{C}^{(0)}(N)$ and $b, d\in \mathcal{A}^{(0)}(N)$.

From above four cases, $\mathcal{AC}(N)$ is $9$-hyperbolic.
\end{proof}

Finally, we prove Theorem~\ref{fourth_thm}.

\begin{proof}[Proof of Theorem~\ref{fourth_thm}]
Fix any geodesic triangle $T=abd$ in $\mathcal{AC}(S)$, where $a$, $b$, and $d$ are vertices of $\mathcal{AC}(S)$.

If $a, b, d\in \mathcal{A}^{(0)}(S)$, then $T$ has a $7$-center in $\mathcal{A}(S)$ by~\cite[Theorem 1.2]{HPW}.
Hence $T$ has a $7$-center in $\mathcal{AC}(S)$.

If $a, b, d\in \mathcal{C}^{(0)}(S)$, then $T$ has a $9$-center in $\mathcal{AC}(S)$ by the proof of~\cite[Theorem 1.1]{HPW}.

If $a\in \mathcal{C}^{(0)}(S)$ and $b, d\in \mathcal{A}^{(0)}(S)$, then we can show that $T$ has an $8$-center in $\mathcal{AC}(S)$ by the same argument that we gave in the proof of Theorem~\ref{third_thm} (see the same case in the proof of Theorem~\ref{third_thm}).

If $a, b\in \mathcal{C}^{(0)}(S)$ and $d$ $\in \mathcal{A}^{(0)}(S)$, then $T$ also has an $8$-center in $\mathcal{AC}(S)$ (see the same case in the proof of Theorem~\ref{third_thm}).

From above four cases, $\mathcal{AC}(S)$ is $9$-hyperbolic.
\end{proof}

\par
{\bf Acknowledgements: } The author would like to express her sincere gratitude to Eiko Kin for her encouragement and invaluable and helpful advices. The author also wish to thank Susumu Hirose, Hideki Miyachi, and Ken'ichi Ohshika for their meticulous comments and advices.


\end{document}